\documentclass[11pt]{article}

\usepackage{amssymb, enumerate}
\usepackage{amsfonts}
\usepackage{amsmath}
\usepackage{color}
\usepackage{hyperref}
\usepackage{epsfig}

\definecolor{green3}{rgb}{0,0.6,0}

\newcommand{\R}{\mathbb R}
\newcommand{\C}{\mathbb{C}}
\newcommand{\Z}{\mathbb{Z}}
\newcommand{\N}{\mathbb{N}}
\newcommand{\Q}{\mathbb{Q}}

\newcommand{\K}{\mathbb{K}}
\renewcommand{\P}{\mathbb{P}}
\newcommand{\LL}{\mathbb{L}}

\newcommand{\cA}{\mathcal{A}}

\newtheorem{proposition}{Proposition}
\newtheorem{theorem}[proposition]{Theorem}

\newtheorem{remark}[proposition]{Remark}
\newtheorem{corollary}[proposition]{Corollary}
\newtheorem{lemma}[proposition]{Lemma}
\newtheorem{example}{Example}

\newenvironment{proof}{
\trivlist \item[\hskip \labelsep\mbox{\it Proof:
}]}{\hfill\mbox{$\square$}
\endtrivlist}

\title{On degree bounds for the sparse Nullstellensatz}
\author{Mar\'\i a Isabel Herrero$^{\natural,\diamondsuit,}$\footnote{Partially supported by the following Argentinean research grants: PIP 11220130100527CO (CONICET) and UBACYT 20020160100039BA (2017-2019).}, Gabriela Jeronimo$^{{ \natural,\flat,\diamondsuit},*}$, Juan Sabia$^{{\flat, \diamondsuit}, *}$}

\begin{document}

\maketitle

\noindent {\small ${\natural}$ Universidad de Buenos Aires. Facultad de Ciencias Exactas y Naturales. Departamento de Matem\'atica. Buenos Aires,  Argentina.}\\
{\small $\flat$ Universidad de Buenos Aires. Ciclo B\'asico Com\'un. Departamento de Ciencias Exactas. Buenos Aires,  Argentina.}\\
{\small ${^\diamondsuit}$ Universidad de Buenos Aires. Consejo Nacional de Investigaciones Cient\'\i ficas y T\'ecnicas. Instituto de Investigaciones Matem\'aticas ``Luis A. Santal\'o'' (IMAS). Buenos Aires, Argentina.}

\bigskip

\noindent E-mails: iherrero@dm.uba.ar, jeronimo@dm.uba.ar, jsabia@dm.uba.ar

\bigskip

\begin{abstract}
We prove new upper bounds for the degrees in Hilbert's Nullstellensatz and for the Noether exponent of polynomial ideals in terms of the monomial structure of the polynomials involved. Our bounds improve the previously known bounds in the sparse setting and are the first to take into account the different supports of the polynomials.
\end{abstract}



\section{Introduction}

Hilbert's Nullstellensatz states that, given polynomials $f_1, \dots, f_s \in k[x_1,\dots, x_n]$ with coefficients in a field $k$ and no common zeroes in the $n$-dimensional affine space over the algebraic closure of $k$, there exist polynomials $g_1, \dots, g_s  \in k[x_1,\dots, x_n]$ such that $1 = \sum_{i=1}^s g_i f_i$. An effective version of this theorem consists in giving upper bounds for the degrees or a characterization of the Newton polytopes of polynomials $g_1, \dots, g_s$ satisfying this identity in terms of the polynomials $f_1, \dots, f_s$.

A great deal of work on the effective Nullstellensatz has been done over the last 30 years (see successive advances in, for example, \cite{Brownawell87}, \cite{Kollar88}, \cite{FitGal90},  \cite{SaSo95},  \cite{Sombra99},  \cite{KPS01}, \cite{Jelonek05}).
The best upper bounds for the degrees of polynomials $g_1,\dots, g_s$ in terms of the degrees of $f_1, \dots, f_s$ are the ones obtained in \cite{Kollar88}, slightly improved in the case $s \le n$  in \cite{Jelonek05}.

When considering particular cases, the degrees of the input polynomials may not be enough for obtaining sharp estimates. A foundational work in this sense is \cite{Bernstein75}, where the Newton polytopes of polynomials and their volumes are introduced to count the number of common zeros  of  polynomial systems (for further development of this theory, see for instance \cite{GKZ94} and \cite{Sturmfels93}). In this sparse setting, the effective Nullstellensatz problem was addressed in \cite{Sombra99}, where the monomial structure of the polynomials $g_if_i$ is characterized leading to upper bounds for the degrees that can improve the previous ones.  The sparse Nullstellensatz from \cite{Sombra99} was dramatically improved in \cite{CanEmi00}, \cite{Tuitman11} and \cite{Wulcan11} for systems with no common roots at toric infinity (for the classical Nullstellensatz, the same behaviour of the degree bounds for generic systems follows from the work of Macaulay, see \cite{Macaulay1916}).
In \cite{CanEmi00}, the authors conjectured that their genericity conditions could be dropped, and this question was addressed in \cite{Massri2016}.

A problem that is closely related to the effective Nullstellensatz consists in estimating the so-called Noether exponent of a polynomial ideal. For an ideal $I\subset k[x_1,\dots, x_n]$, the Noether exponent of $I$ is the minimum integer $\mu$ such that $(\sqrt{I})^\mu \subseteq I$. Given  $f_1,\dots, f_s \in k[x_1,\dots, x_n]$ that generate the ideal $I$, the problem is to find an upper bound for $\mu$ depending on parameters associated to $f_1,\dots, f_s$. Different bounds for this exponent $\mu$  in terms of the degrees and the Newton polytopes of the given generators have been proved (see, for instance, \cite{Kollar88}, \cite{Sombra99}, \cite{Jelonek05}, \cite{Wulcan11}).

In this paper, we prove new bounds for both the degrees in the Nullstellensatz and the Noether exponent of an ideal in the sparse setting. Our work is in the vein of \cite{Jelonek05} and \cite{Sombra99}, in the sense that we consider \emph{arbitrary} sparse systems (that is, no genericity assumptions are made).

First, from the results in \cite{Jelonek05}, we obtain bounds depending on the volume of a convex polytope containing the supports of the given polynomials and the vertex set of the standard unitary simplex (see Propositions \ref{prop:Noetherexpunmixed} and \ref{prop:Nullstunmixed}) which improve the previously known bounds depending on the same invariants.

The main results of the paper are the first upper bounds that distinguish the different supports of the given polynomials. In this setting, we prove upper bounds for the Nullstellensatz (see Theorem \ref{thm:sparseNullsbound} and Corollary \ref{coro:smayorn+1}) and for the Noether exponent of ideals (see Theorem \ref{thm:Noetherexp}). For these mixed sparse systems, our bounds can be considerably smaller than the previous bounds, as illustrated by Examples \ref{ex:MejoraPrevias} and \ref{ex:expNoether}, and, for arbitrary polynomial systems, they differ from them by at most a factor equal to the maximum of the polynomial degrees.

\section{Preliminaries}

Throughout this paper, we work with polynomials with coefficients in a field of characteristic $0$. If $k$ is a field, we write $\overline{k}$ for an algebraic closure of $k$.

Given a finite set $\cA \subset (\Z_{\ge 0})^n$, a \emph{sparse polynomial supported on $\cA$} with coefficients in a field $k$ is a polynomial $f\in k[x_1,\dots, x_n]$ of the form
 $f = \sum_{\alpha\in \cA} c_{\alpha} x^\alpha$ with $c_{\alpha} \in k\setminus\{ 0\}$ for every $\alpha \in \cA$. Here, for $\alpha= (\alpha_1,\dots, \alpha_n) \in (\Z_{\ge 0})^n$, we write $x^\alpha = x_1^{\alpha_1}\dots x_n^{\alpha_n}$.

We say a system of polynomials $f_1,\dots, f_s \in k[x_1,\dots, k_n]$ is an \emph{unmixed} polynomial system supported on $\cA\subset (\Z_{\ge 0})^n$ if, for every $1\le i \le s$, $f_i$ is a polynomial supported on $\cA$. When $f_1,\dots, f_s\in k[x_1,\dots, x_n]$ are supported on (possibly) different subsets $\cA_1,\dots, \cA_s\subset (\Z_{\ge 0})^n$, we say they form a \emph{mixed} sparse polynomial system.
We adopt the notation $V(f_1,\dots, f_s)$ for the set of common zeros of $f_1,\dots, f_s$ in $\overline{k}^n$.

For a family of $n$ finite sets $\cA_1,\dots, \cA_n\subset (\Z_{\ge 0})^n$, we denote $MV_n(\cA_1,\dots, \cA_n)$ the mixed volume of the convex hulls of $\cA_1,\dots, \cA_n$ in $\R^n$ (see, for instance, \cite[Chapter 7]{CLO98} for a definition and basic properties of this notion).
In the case where $\cA_1= \dots = \cA_n= \cA$, we have that $MV_n(\cA_1,\dots, \cA_n) = n! Vol_n(\cA)$, where $Vol_n(\cA)$ is the Euclidean volume in $\R^n$ of the convex hull of $\cA$.
By Bernstein's theorem (see \cite{Bernstein75}), the mixed volume $MV_n(\cA_1,\dots, \cA_n)$ is an upper bound for the number of isolated roots in $(\C\setminus\{ 0\})^n$ of a system of sparse polynomials in $\C[x_1,\dots, x_n]$ supported on $\cA_1,\dots, \cA_n$.

We will write $\Delta_n$ for the vertex set of the standard unitary simplex in $\R^n$, that is, $\Delta_n= \{0, e_1,\dots, e_n\}\subset (\Z_{\ge 0})^n$, where $e_i$ is the $i$th vector of the canonical basis of $\R^n$. Note that $\Delta_n$  is the support set of a generic affine linear polynomial in $n$ variables.

For a finite subset $\cA\subset (\Z_{\ge 0})^n$, we denote by $\mbox{conv}(\cA)$ the convex hull of $\cA$ in $\R^n$. Given a positive integer $m$, we write $m\cdot \cA := m \cdot \mbox{conv}(\cA) = \{ \alpha_1 +\cdots +\alpha_m \mid \alpha_i \in \mbox{conv}(\cA) \ \forall\, 1\le i \le m\}$.
For an integer $1\le r\le n$, we  write $\cA^{(r)}$ for the family consisting of $r$  sets equal to $\cA$.
In particular, we will frequently use the notation $\Delta_n^{(r)}$, which represents the family of supports of $r$ linear forms in $n$ variables. We will also use the notation $\widetilde \cA$ for the set $\widetilde \cA := \{ 0\} \times \cA = \{ (0,\alpha) \in \Z^{n+1} \mid \alpha \in \cA\}$.

\section{The unmixed case}

In this section we will deduce upper bounds for the Noether exponent of an ideal and the Nullstellensatz for unmixed sparse systems by applying the results from \cite{Jelonek05} to suitable toric varieties. With no loss of generality, we work with polynomials with coefficients in an algebraically closed field $\K$ (note that, once a bound is obtained,  over an arbitrary field $k$ all the coefficients of the polynomials involved are solutions to linear systems over $k$).

\begin{proposition}\label{prop:Noetherexpunmixed}
Let $f_1, \dots, f_s \in \K[x_1, \dots, x_n]$ be nonzero polynomials with supports contained in a finite set $\cA \subset (\Z_{\ge0})^n$ and let $I$ be the ideal of $\K[x_1, \dots, x_n]$ generated by $f_1, \dots, f_s$. Then, $(\sqrt{I})^\mu \subset I$ for
$\mu \le n! {Vol}_n(\cA\cup\Delta_n).$
\end{proposition}

\begin{proof}
Let $\cA = \{ \alpha_1,\dots, \alpha_N\}\subset (\Z_{\ge 0})^n$. Following \cite[Theorem 2.10]{Sombra99}, consider the affine toric variety
$$\mathcal{X} = \{ (1,x_1,\dots, x_n, x^{\alpha_1},\dots, x^{\alpha_N}) \mid x\in \K^n\}\subset \K^{1+n+N}.$$
We have that $\mathcal{X}$ may be defined by polynomial equations as
$$\mathcal{X} = \{(y_0,y_1,\dots, y_n, y_{n+1},\dots, y_{n+N}) \in \K^{1+n+N} \mid y_0=1, y_{n+j} = y_1^{\alpha_{j1}}\dots y_n^{\alpha_{jn}}, j=1,\dots, N \},$$
and that it is an irreducible variety of dimension $n$ and degree  $n!\, Vol_n(\cA\cup\Delta_n)$ (by Bernstein's theorem \cite{Bernstein75}).

Consider the map $\varphi\colon\K[y_0,\dots, y_{n+N}] \to \K[x_1,\dots, x_n]$ defined by $\varphi(y_0) = 1$, $\varphi(y_i) = x_i$ for $1\le i\le n$ and $\varphi(y_{n+j}) = x^{\alpha_j}$ for $1\le j \le N$. The kernel of the map $\varphi$ is the defining ideal of the variety $\mathcal{X}$.

Given $f_1,\dots, f_s$ supported on a subset of $\cA$, there exist linear forms $L_1,\dots, L_s\in \K[y_0,\dots, y_{n+N}]$ such that $\varphi(L_i)=  f_i$  for $i=1,\dots, s$ (if $f_i =\sum_{j=1}^N c_{i\alpha_j} x^{\alpha_j}$, we can take $L_i = \sum_{j=1}^N c_{i\alpha_j} y_{n+j}$). Consider the ideal $\mathcal{I}:=(L_1,\dots, L_s) \subset \K[\mathcal{X}]$.
By \cite[Theorem 1.3]{Jelonek05}, $(\sqrt{\mathcal{I}})^D \subset \mathcal{I}$ for $D:= \deg(\mathcal{X}) = n! Vol_n(\cA \cup\Delta_n)$. Therefore, $(\sqrt{I})^D \subset I$.
\end{proof}

\begin{proposition}\label{prop:Nullstunmixed}
Let $f_1, \dots, f_s \in \K[x_1, \dots, x_n]$ be nonzero polynomials with supports contained in a finite set $\cA \subset (\Z_{\ge0})^n$. Let  $d=\max\{\deg(f_i) \mid 1 \le i \le s\}$. If $V(f_1,\dots, f_s)=\emptyset$, there exist polynomials $g_1, \dots, g_s\in \K[x_1, \dots, x_n]$ such that
$$1 = \sum_{i=1}^s g_i f_i \quad \hbox{and} \quad \deg(g_if_i) \le d \, n! \, Vol_n(\cA\cup\Delta_n) \ \hbox{ for every } 1\le i \le s.$$
Moreover, the Newton polytope of $g_i$ is contained in $(n! \, Vol_n(\cA\cup\Delta_n)-1) \cdot \mbox{conv}(\cA \cup \Delta_n)$  for every $1\le i \le s$.
\end{proposition}

\begin{proof} Let $\cA = \{ \alpha_1,\dots, \alpha_N\}\subset (\Z_{\ge 0})^n$. Consider the variety $\mathcal{Y}\subset \P^{n+N}$ defined as the projective closure of the image of the map
$$\psi: \K^n \to \P^{n+N}, \quad \psi(x) = (1:x_1:\dots: x_n: x^{\alpha_1}:\dots: x^{\alpha_N}),$$
which is an irreducible variety of dimension $n$ and degree $n!\, Vol_n(\cA\cup\Delta_n)$.

Assuming $f_i = \sum_{j=1}^N c_{i\alpha_j} x^{\alpha_j}$ for $i=1,\dots, s$, let $L_i= \sum_{j=1}^N c_{i\alpha_j} y_{n+j}$. Due to the fact that $V(f_1,\dots, f_s) =\emptyset$,  we have that $y_0$ lies in the radical of the ideal $(L_1,\dots, L_s) \subset \K[\mathcal{Y}]$.
Applying \cite[Corollary 1.4]{Jelonek05}, we deduce that $y_0^D \in (L_1,\dots, L_s)\subset \K[\mathcal{Y}]$ for $D= \deg(\mathcal{Y}) = n! Vol_n(\cA \cup \Delta_n)$.  Then, there exist homogeneous polynomials $G_1,\dots, G_s \in \K[y_0,\dots, y_{n+N}]$ such that $y_0^D = \sum_{i=1}^s G_i L_i$ $\mod I(\mathcal{Y})$ and $\deg(G_i L_i) = D$ for $i=1,\dots, s$.

Dehomogenizing and taking into account that the polynomials $y_{n+j} y_0^{|\alpha_j|-1}- y_1^{\alpha_{j1}}\dots y_n^{\alpha_{jn}}$, where $|\alpha_j| = \sum_{k=1}^n \alpha_{jk}$, generate the ideal of the affine chart $\mathcal{Y} \cap \{ y_0 \ne 0\}$, we obtain an equality $1= \sum_{i=1}^s g_i f_i$, where $g_i(x) =G_i(1,x_1,\dots, x_n, x^{\alpha_1},\dots, x^{\alpha_N})$ for $i=1,\dots, s$. We conclude that $\deg(g_i f_i)\le \max_{1\le j \le N}\{|\alpha_j|\}. \deg(G_i L_i) \le d D$ and that the Newton polytope of $g_i$ is contained in $(D-1) \mbox{conv}(\cA \cup \Delta_n)$.
\end{proof}

In the following example we can see how these bounds for
the degrees in the Nullstellensatz and the Noether exponent improve the
bounds from \cite{Jelonek05}, \cite{Kollar88}, \cite{KPS01} and \cite{Sombra99}.

\begin{example}\label{ex:Sombra2.12 T3<R4}(see \cite[Example 2.12]{Sombra99} and \cite[Example 4.13]{KPS01}) Let $s\ge 2$, $n\ge 2$, $\delta \ge 2$ and, for $i=1,\dots, s$, let $f_i = a_{i0} +\sum_{j=1}^n a_{ij} x_j+
\sum_{k=1}^{\delta}b_{ik}x_1^{k}\dots x_n^{k}\in \Q[x_1,\dots, x_n]$ be polynomials with supports contained in $\cA = \Delta_n \cup \{ k (e_1+\dots+e_n) ; k=1,\dots, \delta\}$
without common zeros in $\C^n$. Then, by Proposition \ref{prop:Nullstunmixed}, there exist
$g_1, \dots, g_s\in \Q[x_1,\dots, x_n]$ such that
$$\sum_{i=1}^sg_if_i=1\quad  \hbox{and} \quad \deg(g_if_i)\le (n\delta)^2,$$ since
$Vol_n(\cA)= \delta/(n-1)!$ and $\deg(f_i) = n\delta$ for every $1\le i \le s$. Our bound is sharper than the previous ones for these polynomials:
\begin{itemize}
\item the bound from \cite[Corollary 2.11]{Sombra99} is
$\deg(g_if_i ) \le \min\{n+1,s\}^2(n\delta)^2$;
\item from \cite[Corollary 3]{KPS01}, the bound $\deg(g_i) \le 2n^4\delta^2$ is obtained;
\item the bound from \cite[Corollary 1.9]{Kollar88} is $\deg(g_if_i ) \le (n\delta)^{\min\{n,s\}}$;
\item the bound from \cite[Theorem 1.1]{Jelonek05} is $\deg(g_if_i ) \le (n\delta)^{s}$ if $s\le n$ and $\deg(g_if_i ) \le 2 (n\delta)^{n}-1$ if $s>n$.
\end{itemize}

For arbitrary polynomials $f_1,\dots, f_s$ supported on a subset of the previous set $\cA$, the bound for the Noether exponent of the ideal $(f_1,\dots, f_s) \subset \Q[x_1,\dots, x_n]$ from Proposition \ref{prop:Noetherexpunmixed} is $\mu \le n\delta,$ whereas the bound in \cite[Corollary 2.11]{Sombra99} is
$\mu\le \min\{n+1,s\}^2 n \delta $ and
 the bound in \cite[Corollary 1.7]{Kollar88} and \cite[Theorem 1.3]{Jelonek05} is
$\mu\le  (n \delta)^{\min\{n, s\}}$.

Note that, for systems with no zeros at toric infinity, our bounds here would be similar to those obtained in \cite{Tuitman11} and \cite{Wulcan11}, but we do not make any assumptions on the polynomials.
\end{example}

\section{The mixed case}

The aim of this section is to prove upper bounds for the degrees in the Nullstellensatz and for the Noether exponent of ideals generated by mixed sparse systems that take into account the different supports of the given polynomials.
We will follow the approach in \cite{Jelonek05}, which allows us to improve the known bounds but,
unlike the ones in \cite{Sombra99}, \cite{Tuitman11} and \cite{Wulcan11}, does not give a priori estimates of the Newton polytopes of the polynomials involved.

First we prove some auxiliary results.

\begin{lemma}\label{lem:reduced}
Let $\mathbb{L}$ be a field of characteristic zero, $W \in \mathbb{L}[t_0,\dots, t_n]$ be a reduced polynomial  and $D$ a positive integer. If $W(0,t_1, \dots, t_n) \ne 0$, then $W(T_0^D, t_1, \dots, t_n)$ is reduced in $\mathbb{L}[T_0,\dots, t_n]$.
\end{lemma}

\begin{proof}
If $n=0$, the statement follows straightforward because $W(T_0^D)$ has only simple roots in an algebraic closure of $\mathbb{L}$.
For $n \ge 1$,  without loss of generality, it suffices to consider the case when all the irreducible factors of $W$ have positive degree in $t_0$. The proof follows from the univariate case by considering $W$ in $\mathbb{L}(t_1,\dots,t_n)[t_0]$.
\end{proof}

The following proposition can be regarded as a sparse version of \cite[Theorem 3.3]{Jelonek05}, which is, in turn, a generalization of the classical Perron's theorem (see \cite[Satz 57]{Perron27}).
A similar result in the context of implicitization of rational varieties has been proved in \cite{DHM}.

\begin{proposition} \label{prop:degimplicit}
Let $\mathbb{L}$ be a field of characteristic zero, $h_0, \dots, h_n \in \mathbb{L}[x_1, \dots, x_n]\setminus \mathbb{L}$ polynomials with supports $\cA_0,
\dots, \cA_n \subset (\Z_{\ge0})^n$ and $D \in \mathbb{N}$. If the map
$\mathbf{h}:\mathbb{\overline {L}}^n \to \mathbb{\overline{L}}^{n+1}$, $\mathbf{h}(x) =(h_0(x), \dots, h_n(x))$, is generically finite, then there exists a nonzero polynomial $W
\in \mathbb{L}[t_0, \dots, t_n]$ such that $W(h_0, \dots, h_n) =
0$ and
$$ \deg(W(t_0^{D},t_1, \dots, t_n))\le
MV_{n+1}(\widetilde{\cA}_0\cup\{De_{0}\},
\widetilde{\cA}_1\cup\{0,e_{0}\}, \dots,
\widetilde{\cA}_n\cup\{0,e_{0}\}),$$
where, for every $0\le i \le n$, $\widetilde \cA_i := \{ (0,\alpha) \in \Z^{n+1} \mid \alpha \in \cA_i\}$, and $e_0 := (1,0,\dots, 0)$.
\end{proposition}

\begin{proof}
As $\mathbf{h}$ is generically finite, $\overline{\mathbf{h}(\overline{\mathbb{{L}}}^n)}$ is an irreducible hypersurface in $\overline{\mathbb{L}}^{n+1}$. Then, there exists an irreducible polynomial $W \in \mathbb{L}[t_0, \dots, t_n]$ defining it and, therefore, satisfying $W(h_0, \dots, h_n) =
0$.

Let
$P(T_0, t_1 \dots, t_n) = W(T_0^{D},t_1,\dots, t_n) \in \mathbb{L}[T_0,t_1,\dots,t_n]$ and $Y= \{ (y_0,\dots, y_n)  \in \overline{\mathbb{L}}^{n+1}  \mid P(y_0,\dots, y_n) = 0\}$. Note that, by Lemma \ref{lem:reduced}, $P$ is a reduced polynomial and, therefore, $\deg P = \deg Y$.

Consider $\widetilde{Y} = \{ (y,x) \in \overline{\LL}^{n+1} \times  \overline{\LL}^{n} \mid y_0^{D} = h_0(x), y_1=h_1(x), \dots, y_n= h_n(x) \}$. If $\pi_y: \overline{\LL}^{n+1} \times  \overline{\LL}^{n} \to \overline{\LL}^{n+1}$ denotes the projection to the first $n+1$ coordinates, then we will show that
\begin{equation} \label{identity}\overline{\pi_y (\widetilde{Y})} =Y.
\end{equation}
It is clear that $\pi_y (\widetilde{Y}) \subseteq Y$. To prove the converse inclusion, consider a nonzero  polynomial $g \in \LL[t_0, \dots, t_n]$ such that $\overline{\mathbf{h}(\overline{\mathbb{{L}}}^n)} \cap \{ g \ne 0\} \subseteq  \mathbf{h}(\overline{\mathbb{{L}}}^n)$. Note that $P=W(T_0^{D}, t_1,\dots, t_n)$ and $g(T_0^{D}, t_1,\dots, t_n)$ do not have common factors, because $W$ and $g$ in $\LL[t_0, \dots, t_n]$ do not have common factors. Then, the Zariski closure of $ \{y\in \overline{\LL}^{n+1} \mid P(y) =0, g(y_0^{D}, y_1, \dots, y_n)\ne 0\}$ is $Y$.  In addition, for every $y \in  \overline{\LL}^{n+1}$ such that $P(y) = 0$ and $g(y_0^{D}, y_1,\dots, y_n) \ne 0$, we have that $(y_0^{D}, y_1,\dots, y_n) \in \mathbf{h}(\overline{\mathbb{{L}}}^n)$ and, therefore, there exists $x \in \overline{\LL}^{n}$ such that $(y,x) \in \widetilde{Y}$. We conclude identity \eqref{identity} holds.

Now we are going to estimate $\deg (\pi_y(\widetilde{Y}))$. By identity \eqref{identity},  $\pi_y(\widetilde{Y})$ is equidimensional of dimension $n$; then, it suffices to count the number of points in its intersection with a generic linear variety of codimension $n$.  By means of Gaussian elimination, we may assume the linear variety is defined by equations in $\LL [T_0,t_1,\dots, t_n]$ of the form $L_i = t_i + a_i T_0 + b_i$, $i=1,\dots, n$.

Note that there is a one-to-one correspondence between the points $(y_0, \dots, y_n) \in \pi_y(\widetilde{Y} ) \cap \{ L_i=0 ; 1 \le i \le n\}$ and the common zeros $(y_0, x_1, \dots, x_n)$ of the system
$$ h_0(x)-y_0^{D} =0, \ h_1(x) + a_1 y_0 + b_1=0, \dots, h_n(x) + a_n y_0 + b_n=0.$$

Since for generic $a_i, b_i$ $(1\le i \le n)$ the common zeros of this system in $\overline{\LL}^{n+1}$ are isolated and have all nonzero coordinates, by \cite{Bernstein75}, the number of these solutions is bounded by $MV_{n+1}(\widetilde{\cA}_0\cup\{De_{0}\},
\widetilde{\cA}_1\cup\{0,e_{0}\}, \dots,
\widetilde{\cA}_n\cup\{0,e_{0}\}).$
\end{proof}

\subsection{Degree bounds for the Nullstellensatz}

  Our first result in the mixed context is the following upper bound for the degrees in the Nullstellensatz.

\begin{proposition} \label{prop:sparseNulls} Let $\K$ be an
algebraically closed field of characteristic zero, $s \le n+1$ and $f_1, \dots, f_s \in
\K[x_1, \dots, x_n]$ be nonzero polynomials with supports $\cA_1,
\dots, \cA_{s} \subset (\Z_{\ge0})^n$. Let $d=\max\{\deg(f_i) \mid 1 \le i \le s\}$. If $V(f_1,
\dots, f_s)=\emptyset$, there exist polynomials $g_1, \dots, g_s\in \K[x_1, \dots, x_n]$
such that $1 = \sum_{i=1}^s g_i f_i
$ satisfying, for every $1\le i \le s$,
$$\deg(g_if_i) \le d \cdot MV_{n+1}(\widetilde \cA_1 \cup \Delta_{n+1},\dots, \widetilde \cA_s \cup \Delta_{n+1}, \Delta_{n+1}^{(n+1-s)}).$$
\end{proposition}

\begin{proof} We adapt the proof of \cite[Theorem 3.6]{Jelonek05} to our setting. Without loss of generality, we may assume that $f_i\notin \K$ for $i=1,\dots, s$.
Consider the map $\Phi:\K^{n+1}\to \K^{s+n}$, $$\Phi(x_1, \dots, x_n,z)=(zf_1(x), \dots,
zf_s(x),x). $$
As $1 \in (f_1, \dots, f_s)$, this map is one to one and its image $\textrm{Im}(\Phi)$ is a closed subset of $\K^{s+n}$  of dimension $n+1$.
Then, a generic linear projection $\pi:\textrm{Im}(\Phi) \to \K^{n+1}$  is finite and, therefore, it induces a finite morphism $\Psi= \pi \circ \Phi: \K^{n+1} \to \K^{n+1}$.  If $y=(y_1,\dots,y_s)$ and $\pi(y,x) = (L_1(y,x),\dots, L_{n+1}(y,x))$, where $L_1,\dots, L_{n+1}$ are linear forms and their coefficient matrix $A\in \K^{(n+1)\times (s+n)}$  is generic, we may
assume that the determinant of the $(n+1) \times (n+1)$ submatrix of $A$ consisting of its first $n+1$ columns is not zero. Thus, by multiplying by the inverse of this submatrix, we
may assume, without loss of generality, that $\pi (y, x) = (y_1+l_1(x),\dots, y_s+l_s(x),l_{s+1}(x), \dots, l_{n+1}(x))$, and therefore,
$$\Psi(x,z) =
(zf_1(x)+l_1(x),\dots, zf_s(x)+l_s(x),l_{s+1}(x), \dots,l_{n+1}(x)).$$
As $\Psi$ is finite, there exists a minimal polynomial $P$ in $\K[t_1, \dots, t_{n+1}, z]$ monic in $z$ such that $P(\Psi(x,z),z) = 0$.

If $N:= \deg_z (P)$, the coefficient of $z^N$ in the expression $P(\Psi(x,z),z)$ has the form $1-\sum_{i=1}^s g_if_i$.  To estimate the degrees of the polynomials in this expression, note that, for a polynomial $Q\in \K[t_1,\dots, t_{n+1}]$, the degree in the variables $x$ of $Q(\Psi(x,z))$ is at most $\deg_t Q(t_1^{d_1},\dots, t_s^{d_s}, t_{s+1},\dots, t_{n+1})$, where $d_i=\deg(f_i)$ for every $1\le i \le s$. This implies that
\begin{equation}\label{eq:deg_figi}
\deg(g_i f_i) \le \deg_t P(t_1^{d_1},\dots, t_s^{d_s}, t_{s+1},\dots, t_{n+1}, z),
\end{equation}
Then, it suffices to obtain an upper bound for the degree in the variables $t=(t_1,\dots, t_{n+1})$ of this polynomial.

In order to do so, consider the field $\mathbb{L}=\K(z)$ and apply Proposition \ref{prop:degimplicit} to $D=1$ and the polynomials $zf_1(x)+l_1(x),\dots, zf_s(x)+l_s(x),l_{s+1}(x), \dots,l_{n+1}(x)\in \mathbb{L}[x]$, which induce a generically finite map $\widehat\Psi : \overline{\mathbb{L}}^n \to \overline{\mathbb{L}}^{n+1}$. Since,  for every $1\le i \le s$, the support of $zf_i(x) +l_i(x)$ is $\cA_i \cup \Delta_n$ and, for every $s+1\le i \le n+1$, the support of $l_i(x)$ is $\Delta_n$, it follows that there exists a nonzero polynomial $W \in \mathbb{L}[t_1,\dots, t_{n+1}]$ such that $W(\widehat \Psi (x)) = 0$ and
\begin{equation}\label{eq:degtW}
\deg_t (W) \le MV_{n+1}(\widetilde \cA_1 \cup \Delta_{n+1},\dots, \widetilde \cA_s \cup \Delta_{n+1}, \Delta_{n+1}^{(n+1-s)}).
\end{equation}
By clearing denominators, we may assume that $W\in \K[t_1,\dots, t_{n+1},z]$. The minimality of $P$ implies that $P$ divides $W$ and so,
\begin{equation} \label{eq:degdiv}
\deg_t P(t_1^{d_1},\dots, t_s^{d_s}, t_{s+1}, \dots, t_{n+1}, z) \le \deg_t W(t_1^{d_1},\dots, t_s^{d_s}, t_{s+1}, \dots, t_{n+1},z)\le d\cdot \deg_t( W).
\end{equation}
The result follows from inequalities \eqref{eq:deg_figi}, \eqref{eq:degdiv} and \eqref{eq:degtW}.
\end{proof}

When the number of polynomials involved is at most $n$, the previous bound can be rewritten in terms of an $n$-dimensional mixed volume:

\begin{remark}\label{rem:identityMV} If $s\le n$, it is not difficult to see that
$$MV_{n+1}(\widetilde \cA_1 \cup \Delta_{n+1},\dots, \widetilde
\cA_s \cup \Delta_{n+1}, \Delta_{n+1}^{(n+1-s)})= MV_{n}(\cA_1 \cup
\Delta_{n},\dots, \cA_s \cup \Delta_{n}, \Delta_{n}^{(n-s)})$$
and then, in this case, the bound stated in Proposition \ref{prop:sparseNulls} can be re-written as
\begin{equation}\label{eq:ndimsparsebound}
\deg(g_if_i) \le d \cdot
MV_{n}(\cA_1 \cup \Delta_{n},\dots, \cA_s \cup \Delta_{n}, \Delta_{n}^{(n-s)}).
\end{equation}
\end{remark}

Note that in the particular case of a polynomial system $f_1, \dots, f_n, f_{n+1} \in \Q[x_1,\dots, x_n]$ where $f_i= x_i - a_i$ for $i=1,\dots, n$, and $f_{n+1}$ is a polynomial with support $d \Delta_n$ such that $f_{n+1} (a_1,\dots, a_n)\ne 0$, the bound for the degrees given in Proposition \ref{prop:sparseNulls} is $d^2$. However, it is easy to write $1= \sum_{i=1}^{n+1} g_i f_i$  with $\deg(g_if_i) \le d$ for every $1\le i\le n+1$.  This is in fact the well-known degree bound in the Nullstellensatz in terms of the degrees of the polynomials (see \cite{Jelonek05}).
However, we may obtain another bound for the degrees in the sparse Nullstellensatz, which enables us to deduce a refinement of our previous result, by applying Proposition \ref{prop:degimplicit} in a different way.

\begin{proposition} \label{prop:refsparseNulls}
With the same assumptions and notation as in Proposition \ref{prop:sparseNulls}, for every $1\le j \le s$, let $d_j := \deg(f_j)$, $\delta_j :=\max\{d_i \mid 1\le i \le s, i\ne j\}$ and $M_j := MV_{n}( \cA_1 \cup \Delta_{n},\dots,  \cA_{j-1} \cup \Delta_{n}, \cA_{j+1} \cup \Delta_{n},\dots,\cA_s \cup \Delta_{n}, \Delta_{n}^{(n+1-s)})$.
Then, for every $1\le i \le s$, we have:
$$ \deg(g_if_i)\le \min_{1\le j \le s} \{d_j \delta_{j} M_j \}.
$$
\end{proposition}

\begin{proof}
 By taking $D:=d_1$ in the statement of Proposition \ref{prop:degimplicit}, we deduce that the polynomial $W$ appearing in the proof of Proposition \ref{prop:sparseNulls} satisfies
$$
\deg_t W(t_1^{d_1},t_2, \dots, t_{n+1},z) \le MV_{n+1}(\widetilde \cA_1 \cup \widetilde \Delta_{n} \cup \{d_1e_0\},\widetilde \cA_2 \cup \Delta_{n+1},\dots, \widetilde \cA_s \cup \Delta_{n+1}, \Delta_{n+1}^{(n+1-s)}).
$$
Taking into account that $\widetilde \cA_1 \cup \widetilde \Delta_{n} \cup \{d_1e_0\}\subset d_1 \Delta_{n+1}$, by basic properties of mixed volumes,
we obtain:
\begin{eqnarray*}
\deg_t W(t_1^{d_1},t_2, \dots, t_{n+1},z) &\le& d_1 MV_{n+1}(\widetilde \cA_2 \cup \Delta_{n+1},\dots, \widetilde \cA_s \cup \Delta_{n+1}, \Delta_{n+1}^{(n+2-s)}) \\
&=& d_1 MV_{n}(\cA_2 \cup \Delta_{n},\dots, \cA_s \cup \Delta_{n}, \Delta_{n}^{(n+1-s)}) \ =  \ d_1 M_1
\end{eqnarray*}
Then, if $\delta_{1}:=\max \{d_2,\dots, d_s\}$, we may replace inequality \eqref{eq:degdiv} with
\begin{eqnarray*}
\deg_t P(t_1^{d_1},\dots, t_s^{d_s}, t_{s+1}, \dots, t_{n+1}, z) &\le&
\deg_t W(t_1^{d_1},\dots, t_s^{d_s}, t_{s+1}, \dots, t_{n+1}, z)\\
&\le & \delta_{1} \cdot \deg_t W(t_1^{d_1},t_2, \dots, t_{n+1},z) \\
&\le & d_1 \delta_{1} M_1.
\end{eqnarray*}
By interchanging the roles of $f_1$ and $f_j$ for every $1\le j \le s$,  we deduce from \eqref{eq:deg_figi} the stated upper bound.
\end{proof}

Combining the results in Propositions \ref{prop:sparseNulls} and \ref{prop:refsparseNulls}, we deduce:

\begin{theorem} \label{thm:sparseNullsbound} Let $\K$ be an
algebraically closed field of characteristic zero, $s \le n+1$ and $f_1, \dots, f_s \in
\K[x_1, \dots, x_n]$ be nonzero polynomials with supports $\cA_1,
\dots, \cA_{s} \subset (\Z_{\ge0})^n$. Let $d:=\max\{\deg(f_i) \mid 1 \le i \le s\}$, $M:=MV_{n+1}(\widetilde \cA_1 \cup \Delta_{n+1},\dots, \widetilde \cA_s \cup \Delta_{n+1}, \Delta_{n+1}^{(n+1-s)})$ and, for $1\le j \le s$, let $d_j:=\deg(f_j)$, $\delta_j := \max\{ d_i \mid 1\le i \le s, \, i\ne j\}$ and $M_j := MV_{n}( \cA_1 \cup \Delta_{n},\dots,  \cA_{j-1} \cup \Delta_{n}, \cA_{j+1} \cup \Delta_{n},\dots,\cA_s \cup \Delta_{n}, \Delta_{n}^{(n+1-s)})$. If $V(f_1, \dots, f_s)=\emptyset$, there exist polynomials $g_1, \dots, g_s\in \K[x_1, \dots, x_n]$ such that $1 = \sum_{i=1}^s g_i f_i
$ satisfying, for every $1\le i \le s$,
$
\deg(g_if_i) \le N(\cA_1,\dots, \cA_s; n),
$
where $$N(\cA_1,\dots, \cA_s; n):=\min\{ d M ;  d_j \delta_j M_j, 1\le j \le s\}.$$
\end{theorem}

The following example illustrates how our bound for
the degrees in the Nullstellensatz for mixed sparse systems may  improve the
bounds from \cite{Kollar88}, \cite{Jelonek05}, \cite{KPS01} and \cite{Sombra99} considerably.

\begin{example}\label{ex:MejoraPrevias} Let $d\in \N$, $d\ge 2$.
For every $ 1 \le i \le n$, consider a polynomial $f_i \in
\Q[x_1, \dots, x_n]$ with support $\cA = \Delta_n \cup \{ 2e_1,\dots, de_1\}$, $f_i = a_{i0} + \sum_{j=1}^n a_{ij} x_j +
\sum_{k=2}^d b_{ik} x_1^k$, and let
$f_{n+1} \in \Q[x_1, \dots, x_n]$ be a polynomial with support $\cA_{n+1} =
d\cdot \Delta_n\subset (\Z_{\ge 0})^n$. Assume $f_1,\dots, f_n,f_{n+1}$ do not have common zeros in $\C^n$.
Then, by Theorem \ref{thm:sparseNullsbound}, there exist polynomials $g_1,\dots, g_n, g_{n+1}\in \Q[x_1,\dots, x_n]$ such that
$$\sum_{i=1}^{n+1} g_if_i=1\quad  \hbox{and} \quad \deg(g_if_i)\le d^3,$$
since $MV_{n+1}((\widetilde \cA \cup\Delta_{n+1})^{(n)}, \widetilde \cA_{n+1}\cup \Delta_{n+1}) = d^2$,
$MV_n( (\cA \cup\Delta_{n})^{(n)}) = d$, $MV_n( (\cA \cup\Delta_{n})^{(n-1)}, \cA_{n+1}\cup\Delta_n) = d^2$ and $\deg(f_i ) = d$ for every $1\le i \le n+1$. For this system,
\begin{itemize}
\item the bound in \cite[Corollary 1.9]{Kollar88} is
$\deg(g_if_i)\le d^n$, assuming $d\ge 3$;
\item the bound in \cite[Theorem 1.1]{Jelonek05} is
$\deg(g_if_i)\le 2d^n-1$;
\item the bound in \cite[Theorem 3.19]{Sombra99} is  $\deg(g_i f_i) \le 2d^n$;
\item the bound from \cite[Corollary 4.11]{KPS01} is $\deg(g_i)\le 2n^2d^{n}$.
\end{itemize}
\end{example}

\begin{remark}
The polynomials obtained by homogeneizing those in Example \ref{ex:MejoraPrevias} define a projective variety of codimension $2$ contained in the hyperplane at infinity. This would imply that the bound in \cite[Corollary 1.3]{AndWul11} (see also \cite{EinLaz99}), which takes into account distinguished components at infinity, is similar to our bound. More generally, it would be interesting to generalize these results to the mixed sparse setting.
\end{remark}

{}From Theorem \ref{thm:sparseNullsbound}, we can also deduce a bound for the degrees in the Nullstellensatz for a family of $s>n+1$ sparse polynomials.
Let $f_1,\dots, f_s\in \K[x_1,\dots, x_n]$ be nonzero polynomials with supports $\cA_1,\dots, \cA_s \subset (\Z_{\ge 0})^n$ such that $V(f_1,\dots, f_s) = \emptyset$. By taking generic linear combinations of $f_1,\dots, f_s$, for a set  $J = \{j_1,\dots, j_{n+1}\}$ with $1\le j_1<\dots<j_{n+1} \le s$, we can obtain polynomials $h_1,\dots, h_{n+1}$ with supports
\begin{equation}\label{eq:suppI}
\cA_{J,1}:=\cA_{j_1} \cup \bigcup_{i\notin J}\cA_i, \dots, \cA_{J,n+1}:=\cA_{j_{n+1}} \cup \bigcup_{i\notin J}\cA_i
\end{equation}
such that $V(h_1,\dots, h_{n+1} )= \emptyset$. By Theorem \ref{thm:sparseNullsbound}, there exists polynomials $g_{J,1},\dots, g_{J, n+1}\in \K[x_1,\dots, x_n]$ such that $1 = \sum_{i=1}^{n+1} g_{J, i} h_i$ and $\deg(g_{J, i} h_i) \le N(\cA_{J,1}, \dots, \cA_{J,n+1}; n)$ for every $1\le i \le n+1$. It follows that there exist polynomials $g_1,\dots, g_s \in \K[x_1,\dots, x_n]$ such that $1 = \sum_{i=1}^s g_i f_i$ satisfying $\deg(g_i) \le   N(\cA_{J,1}, \dots, \cA_{J,n+1}; n)$ for every $1\le i \le s$. Therefore, we have:

\begin{corollary}\label{coro:smayorn+1}
Let $s>n+1$ and $f_1,\dots, f_s\in \K[x_1,\dots, x_n]$ be nonzero polynomials with supports $\cA_1,\dots, \cA_s \subset (\Z_{\ge 0})^n$ such that $V(f_1,\dots, f_s) = \emptyset$. Then, using the notation in Theorem \ref{thm:sparseNullsbound}, there exist polynomials $g_1,\dots, g_s \in \K[x_1,\dots, x_n]$ such that $1=\sum_{i=1}^s g_i f_i$ satisfying, for every $1\le i \le s$,
$$\deg(g_i) \le \min_{J \subset \{1,\dots, s\}\atop |J| = n+1} N(\cA_{J, 1}, \dots, \cA_{J, n+1}; n),$$
where $\cA_{J,1},\dots, \cA_{J,n+1}$ are the sets defined in \eqref{eq:suppI}.
\end{corollary}

\subsection{Noether exponent}

This section is devoted to proving an upper bound for the Noether exponent of a polynomial ideal generated by a mixed system in terms of the supports of the given generators.
We first prove a suitable sparse version of the Generalized Elimination Theorem from \cite[Theorem 4.3]{Jelonek05}.

\begin{proposition}\label{prop:elimination} Let $s \le n$ and $f_1, \dots, f_s \in \K[x_1, \dots, x_n]$ be nonzero polynomials with
supports $\cA_1, \dots, \cA_{s} \subset (\Z_{\ge0})^n$. Let
$d=\max\{\deg(f_i)\mid 1 \le i \le s\}$. Assume $V(f_1, \dots, f_s) \ne \emptyset$. If $G\in \K[x_1, \dots, x_n]$ is a nonzero polynomial
which is constant over every irreducible component of $V(f_1, \dots, f_s)$,
there exist polynomials $g_1, \dots, g_s \in \K[x_1, \dots,
x_n]$ and a nonzero univariate polynomial $\phi\in \K[T]$ such that
$$\phi(G) = \sum_{i=1}^s g_i f_i \ \hbox{ and } \ \deg(g_if_i) \le \deg(G) \cdot d \cdot MV_{n}(\cA_1\cup\Delta_{n}, \dots, \cA_s\cup
\Delta_{n},\Delta_{n}^{(n-s)}).$$
\end{proposition}

\begin{proof}
Let $\Phi\colon \K^{n+1} \to \K^{s+n}$, $\Phi(x,z) = (zf_1(x),\dots, zf_s(x), x)$. Since $G$ is constant over each irreducible component of $V:=V(f_1,\dots, f_s)$, we have that $G(V)$ is a finite set $\{a_1,\dots, a_q\}$. Consider the polynomial $Q= \prod_{1\le i \le q} (T-a_i)\in \K[T]$. Then,  $Q(G) \in \sqrt{(f_1,\dots, f_s)}$ and, therefore, localizing in the multiplicative set of the powers of $Q(G)$, we have that $1\in (f_1,\dots, f_s)_{Q(G)} \subset \K[x_1,\dots, x_n]_{Q(G)}$. It follows that the map $\Phi$ is one to one outside the zero set of $Q(G)$.

Set $\Gamma \subset \K^{s+n}$ for the Zariski closure of $\mbox{Im}(\Phi)$, which is an $(n+1)$-irreducible variety.
Then, for generic linear forms $\ell_0,\ell_1,\dots, \ell_n$ in $\K[y_1,\dots,y_s,x_1,\dots, x_n]$, the projection
$\pi_\ell : \Gamma \to \K^{n+1}$, $\pi_\ell(y, x) = (\ell_0(y,x), \dots, \ell_n(y,x))$, is a finite morphism.

Consider now $\pi_{\ell, G}: \Gamma \to \K^{n+2}$, $\pi_{\ell, G} (y,x) = (\ell_0(y,x), \dots, \ell_n(y,x), G(x))$.
There is a nonzero polynomial $P\in \K[t_0,\dots, t_n, T]$ such that $P(\ell_0(y,x), \dots, \ell_n(y,x), G(x)) =0$.
By making a change of variables $T_i = t_i - \alpha_i t_0$, for $i=1,\dots, n$, and $T_0 = t_0$, if $L_i = \ell_i - \alpha_i \ell_0$, we obtain that
$P(\ell_0, L_1 +\alpha_1 \ell_0, \dots, L_n +\alpha_n \ell_0, G) = \rho(G) \ell_0^D + \sum_{j=1}^D A_j(L_1,\dots, L_n, G) \ell_0^{D-j}
 = \widetilde P (\ell_0,L_1,\dots, L_n,G)$,
where $\widetilde P \in \K[T_0,\dots, T_n, T]$ is a nonzero polynomial whose leading coefficient as a polynomial in $T_0$ equals $\rho(T)$, which depends only on the variable $T$.
Thus, we obtain a finite morphism
$\K[L_1,\dots,L_n, G]_{\rho(G)} \to \K[\ell_0,L_1,\dots, L_n,G]_{\rho(G)}$.

Recalling that we also have a finite map $\K[\ell_0,L_1,\dots,L_n]=\K[\ell_0,\ell_1,\dots, \ell_n]\to \K[\Gamma ]$ and  a bijection $\K[\Gamma]_{Q(G)} \to \K[x_1,\dots, x_n,z]_{Q(G)}$, we deduce that the induced composition
$$\K[L_1,\dots, L_n, G]_{Q(G)\rho(G)}\to \K[\Gamma]_{Q(G)\rho(G)}\to  \K[x_1,\dots, x_n,z]_{Q(G) \rho(G)}$$
is finite.
Without loss of generality, we may assume that, for $i=1,\dots, s$,
$L_i(y,x) = y_i + \mu_i(x)$ for a generic linear form $\mu_i\in \K[x]$, and that, for $i=s+1,\dots, n$, $L_i$ depends only on the variables $x$.

Let $P_z \in \K[T_1,\dots, T_n,T][Z]$ be a minimal polynomial of $z$; then, $$P_z(zf_1(x) +\mu_1(x),\dots, zf_s(x)+\mu_s(x), L_{s+1}(x),\dots, L_n(x), z) = 0$$ and the leading coefficient $\phi$ of $P_z$ is a factor of a power of $Q(T)\rho(T)$.
The proof finishes similarly as the proof of Proposition \ref{prop:sparseNulls}, by considering the coefficient of $z^{\deg_z(P_z)}$ in the expansion of $P_z(zf_1(x) +\mu_1(x),\dots, zf_s(x)+\mu_s(x), L_{s+1}(x),\dots, L_n(x), z)$, which is of the form $\phi(G) - \sum_{i=1}^{s} g_i(x) f_i(x)$. Considering $P_z$ as a polynomial in $\K(z)[T_1,\dots, T_n, T]$ and applying Proposition \ref{prop:degimplicit},  we deduce that
$$\deg P_z(T_1^{d_1},\dots, T_s^{d_s}, T_{s+1},\dots, T_n, T^{\deg(G)}) \le d  \cdot \deg P_z(T_1,\dots, T_n, T^{\deg(G)})$$
$$\le d \cdot  MV_{n+1}(\widetilde \cA_1 \cup\Delta_{n+1},\dots, \widetilde \cA_s \cup \Delta_{n+1}, \deg(G)\Delta_{n+1}, \Delta_{n+1}^{(n-s)})$$
$$\le \deg(G) \cdot d \cdot MV_{n+1}(\widetilde \cA_1 \cup\Delta_{n+1},\dots, \widetilde \cA_s \cup \Delta_{n+1}, \Delta_{n+1}^{(n+1-s)}).$$
$$ {} = \deg(G) \cdot d \cdot MV_{n}(\cA_1\cup\Delta_{n}, \dots, \cA_s\cup
\Delta_{n},\Delta_{n}^{(n-s)}).$$
\end{proof}

The main result of this section is the following.

\begin{theorem}\label{thm:Noetherexp}
Let $f_1, \dots, f_s \in \K[x_1, \dots, x_n]$ be nonzero polynomials with supports $\cA_1,
\dots, \cA_{s} \subset (\Z_{\ge0})^n$ and $d=\max\{\deg(f_i)\mid
1 \le i \le s\}$. Let $I$ be the ideal of $\K[x_1, \dots, x_n]$ generated by $f_1, \dots,
f_s$. Then, $(\sqrt{I})^\mu \subset I$ for
$$\mu \le d\cdot MV_{n}(\cA_1\cup\Delta_n, \dots, \cA_s\cup
\Delta_n, \Delta_n^{(n-s)}) \hbox{ if } s\le n,$$
and
$$\mu \le d \cdot \min_{J\subset\{1,\dots, s\} \atop |J|=n} \{ MV_n(\cA_{J,1}\cup \Delta_n,\dots, \cA_{J, n} \cup \Delta_n)\} \hbox{ if } s\ge n+1,$$
where, for $J= \{j_1,\dots, j_n\}$, $\cA_{J,i}:=\cA_{j_i} \cup \bigcup_{k\notin J}\cA_k$ for every $1\le i \le n$.
\end{theorem}

\begin{proof}
Assume $I\ne (1)$, since otherwise there is nothing to prove.

First, we consider the case $s\le n$. Let $G\in \sqrt{I}$.
By Proposition \ref{prop:elimination}, there exist $\phi\in
\K[T]\setminus\{ 0 \}$  such that
$\phi(G) \in I$ and $\deg(\phi(G)) \le \deg(G) \cdot d \cdot MV_{n}(\cA_1\cup\Delta_{n}, \dots, \cA_s\cup
\Delta_{n},\Delta_{n}^{(n-s)})$.
Write $\phi(T) = T^{\mu_G} \prod_{j=1}^{r}(T-a_j)$ with $a_j \in \K\setminus\{0\}$. Since, for $1\le j \le r$,  $G-a_j$ does not lie in any associated prime of the ideal $I$ (because $G$ lies in all of them), the fact that $G^{\mu_G}\prod_{j=1}^{r}(G-a_j)\in I$ implies that $G^{\mu_G} \in I$. Now,
$$
\mu_G \cdot \deg(G) \le \deg (\phi(G)) \le  \deg(G) \cdot d \cdot MV_{n}(\cA_1\cup\Delta_{n}, \dots, \cA_s\cup
\Delta_{n},\Delta_{n}^{(n-s)}),
$$
and, therefore, $\mu_G \le d\cdot MV_{n}(\cA_1\cup\Delta_{n}, \dots, \cA_s\cup
\Delta_{n},\Delta_{n}^{(n-s)})$.

We conclude that there exists $\mu\in \Z_{\ge 0}$, $\mu\le  d\cdot MV_{n}(\cA_1\cup\Delta_{n}, \dots, \cA_s\cup
\Delta_{n},\Delta_{n}^{(n-s)})$, such that $G^\mu\in I$ for every $G\in \sqrt{I}$. From this fact, it is easy to prove that $(\sqrt{I}) ^\mu \subset I$ (see, for instance, the proof of \cite[Corollary 4.6]{Jelonek05}).

Assume now that $s\ge n+1$. By taking generic linear combinations of $f_1,\dots, f_s$, for a set  $J = \{j_1,\dots, j_{n}\}$ with $1\le j_1<\dots<j_{n} \le s$, we can obtain polynomials $h_1,\dots, h_{n}$ with supports $\cA_{J,1},\dots, \cA_{J, n}$
such that $V(h_1,\dots, h_{n} )\setminus V(f_1,\dots, f_s)$ is a finite set. Then, if $G\in \sqrt{I}$, it satisfies the assumptions in Proposition \ref{prop:elimination} for the polynomials $h_1,\dots, h_n$. Therefore, there exists $\phi\in \K[T]\setminus \{ 0\}$ such that $\phi(G) \in (h_1,\dots, h_n)$ and, as a consequence, $\phi(G) \in I$, with $\deg(\phi(G))\le \deg(G) \cdot d \cdot  MV_n(\cA_{J,1}\cup \Delta_n,\dots, \cA_{J, n} \cup \Delta_n)$. As before, it follows that $\mu \le d \cdot  MV_n(\cA_{J,1}\cup \Delta_n,\dots, \cA_{J, n} \cup \Delta_n)$.
\end{proof}

Finally, we present an example that compares our mixed bound for the Noether exponent to previously known bounds:

\begin{example}\label{ex:expNoether} Consider positive integers $D$ and  $1 \le D_1 \le \dots \le D_n$.  Let $\cA = \Delta_n\cup \{k(e_1+\dots +e_n); 1\le k \le D\}\subset (\Z_{\ge 0})^n$.  For an ideal $I= (f_1,\dots, f_n)$ in $\Q[x_1,\dots, x_n]$ generated by polynomials $f_1, \dots, f_n$ with supports $\cA_i = D_i\cdot \cA $, for all $ 1 \le i \le
n$, by Theorem \ref{thm:Noetherexp}, we have that $(\sqrt{I})^\mu \subset I$ for a non-negative integer $$\mu\le \Big(\prod_{i=1}^n D_i\Big) n^2 D^2 D_n,$$ since $MV_n(\cA_1,\dots, \cA_n) = \left(\prod_{i=1}^n D_i\right) MV_n(\cA^{(n)}) = \left(\prod_{i=1}^n D_i\right) nD$ and, for $1\le i \le n$, $\deg(f_i) = n D D_i$. On the other hand, in this case,
\begin{itemize}
\item the bound in \cite[Corollary 2.11]{Sombra99} is
$ n^3 D  D_n^n$;
\item the bound in \cite[Theorem 1.3]{Jelonek05} is
$ \Big(\prod_{i=1}^nD_i\Big) n^n D^n$.
\end{itemize}
In particular, when $D_1 = \dots = D_{n-1} = 1$ and $D_n = D$, our bound is $n^2 D^4$, whereas the bounds from \cite{Sombra99} and \cite{Jelonek05}  are $n^3 D^{n+1}$ and $n^n D^{n+1}$,  respectively.
\end{example}

\bigskip
\noindent \textbf{Acknowledgments.} The authors wish to thank the anonymous referees for their helpful comments.

\end{document}